\documentclass[pdftex,11pt,a4paper,twoside]{article}
\usepackage[german,english]{babel}
\usepackage[]{graphicx} 
\usepackage{subfig}		
\usepackage[colorlinks=false]{hyperref}   
\usepackage{amsmath,amssymb,amsthm,amsfonts}
\usepackage{mathtools}
\usepackage{paralist} 
\usepackage{enumitem} 
\usepackage{bbm} 
\usepackage[numbers,comma,sort&compress]{natbib}

\usepackage{setspace}
\usepackage{vmargin}
\setmarginsrb  { 1.5in}  
{ 0.6in}  
{ 1.0in}  
{ 0.8in}  
{  20pt}  
{0.3in}  
{   9pt}  
{ 0.3in}  
\usepackage{fancyhdr}
\usepackage{url}

\makeatletter
\def\url@leostyle{%
  \@ifundefined{selectfont}{\def\UrlFont{\sf}}{\def\UrlFont{\small\ttfamily}}}
\makeatother
\numberwithin{equation}{section}
\newtheoremstyle{myDef}%
{} 
{} 
{\itshape} 
{} 
{\bfseries} 
{.} 
{0.5em} 
{} 
\theoremstyle{myDef}
\newtheorem{theorem}{Theorem}[section]

\newtheorem{prop}[theorem]{Proposition}

\newcommand{\fdef}{\textbf} 
\newcommand{\falg}[1]{\texttt{\textsc{#1}}} 

\newcommand{\slope}[2][{z, \x{x}}]{\ensuremath{{#2}\!\left[#1\right]}}
\newcommand{\that}{\ensuremath{\hat{x}}}
\newcommand{\zz}[1]{\ensuremath{\mathbb{#1}}} 
\newcommand{\assign}{\ensuremath{\mathrel{\mathop:}=}} 
\newcommand{\id}{\mathbbm{1}}
\newcommand{\lbar}[1]{\ensuremath{\underline{#1}}}
\newcommand{\ubar}[1]{\ensuremath{\overline{#1}}}
\newcommand{\intval}[1]{\ensuremath{\mathbf{#1}}}

\newcommand{\til}[1]{\ensuremath{\widetilde{#1}}}
\newcommand{\absv}[1]{\ensuremath{\left\vert{#1}\right\vert}}
\DeclareMathOperator{\interior}{int}

\usepackage{caption}
\DeclareCaptionSubType[alph]{figure} 

\hypersetup{%
  pdftitle = {Exclusion regions for parameter-dependent systems of equations},
  pdfsubject = {},
  pdfauthor = {Bettina Ponleitner, Hermann Schichl}, 
  bookmarksnumbered=true%
}

\begin{document}
  \title{\textbf{Exclusion regions for parameter-dependent systems of equations}}
  \author{Bettina Ponleitner, Hermann Schichl}
  \date{Faculty of Mathematics, University of Vienna}
  \maketitle

  \section*{Abstract}
This paper presents a new algorithm based on interval methods for rigorously constructing inner estimates of feasible parameter regions together with enclosures of the solution set for parameter-dependent systems of nonlinear equations in low (parameter) dimensions.  The proposed method allows to explicitly construct feasible parameter sets around a regular parameter value, and to rigorously enclose a particular solution curve (resp.~manifold) by a union of inclusion regions, simultaneously. The method is based on the calculation of inclusion and exclusion regions for zeros of square nonlinear systems of equations. Starting from an approximate solution at a fixed set $p$ of parameters, the new method provides an algorithmic concept on how to construct a box $\intval{s}$ around $p$ such that for each element $s\in \intval{s}$ in the box the existence of a solution can be proved within certain error bounds.

\section{Introduction}
In this paper we consider parameter-dependent nonlinear systems of equations
\begin{equation}\label{eq:H}
H(x,s) = 0,
\end{equation}
with solutions $x = x(s) \in \zz{R}^n$ depending on a set of parameters $s\in\zz{R}^p$. Thereby, we assume $H\colon X\times S \subseteq \zz{R}^n\times \zz{R}^p\to \zz{R}^n$ to be differentiable with Lipschitz continuous derivative.

Branch and bound methods for finding all zeros of a (square) nonlinear system of equations in a box frequently have the difficulty that subboxes containing no solution cannot be eliminated if there is a nearby zero outside the box. This results in the so-called \emph{cluster effect}, i.e., the creation of many small boxes by repeated splitting, whose processing may dominate the total work spent on the global search. \citet{schichl2005ex} presented a method how to reduce the cluster effect for nonlinear $n\times n$-systems of equations by computing so-called \emph{inclusion} and \emph{exclusion regions} around an approximate zero with the property, that a true solution lies in the inclusion region and no other solution in the corresponding exclusion region, which thus can be safely discarded. 

 In the parameter-dependent case, it would be convenient to show the existence of such inclusion regions for a whole set of parameter values in order to rigorously identify feasible parameter boxes $\intval{s}$ where for all $s\in \intval{s}$ solutions $x(s)$ of \eqref{eq:H} exist. Thus, we extend the method from \citet{schichl2005ex} to this problem class, and show how to compute parameter boxes $\intval{s}\subseteq S$ such that for each parameter set $s \in \intval{s}$ the existence of a solution  $x(s) \in \intval{x}\subseteq X$ of \eqref{eq:H} within a narrow inclusion box can be guaranteed.

     The procedure for computing such a feasible parameter box $\intval{s}$ consists of three main steps:
   \begin{inparaenum}[(i)]
     \item solve \eqref{eq:H} for a fixed parameter $p\in\intval{s}$ and compute a pair of  inclusion and exclusion regions for a corresponding approximate zero $z\approx x(p)$ as described in \cite{schichl2005ex},   \label{i}
     \item consider an approximation function $\that(s)\colon S\to X$ for the solution curve, and 
     \item extend the estimates and bounds from step (\ref{i}) using slope forms in order to calculate  a feasible parameter box $\intval{s}$ around $p$ such that for all $s\in\intval{s}$ the existence of a solution $x^*(s)$ of \eqref{eq:H} can be proved. 
   \end{inparaenum}
   
      \paragraph{Other known approaches.} Parameter-dependent systems of equations can be solved by continuation methods (e.g., \cite{allgower1985, allgower1997}) which trace a particular solution curve or a solution manifold, if $p>1$ in \eqref{eq:H}. Another approach for parametric polynomial systems is to use Gr\"obner bases (e.g., \cite{kapur1995, montes2010}). \citet[Thm.~5.1.3]{neumaier1991} formulated a semilocal version of the implicit function theorem and provided a tool \cite[Prop.~5.5.2]{neumaier1991} for constructing an enclosure of the solution set of \eqref{eq:H} with parameters varying in narrow intervals. Furthermore, \citet{neumaier1989} performed a rigorous sensitivity analysis for parameter-dependent systems of equations and proved a quadratic approximation property of a slope based enclosure. \citet{kolev2001} proposed an iterative method to construct a linear interval enclosure of the solution set of \eqref{eq:H} over a given parameter interval. \citet{goldsztejn2003} used a weak version of the parametric Miranda-theorem (see \cite[Thm.~5.3.7]{neumaier1991}) to verify the existence of solutions over a given parameter interval and to compute a reliable inner estimate of the feasible parameter region. Independently from the work of \citeauthor{goldsztejn2003}, the authors recently pursued a similar approach and propose some tools utilizing Miranda's theorem and centered forms for rigorously solving parameter-dependent systems of equations \cite{ponleitnerSchichl2018mir}.

   \paragraph{Outline.} The paper is organized as follows: in Section \ref{sec:slopeForms} we review some central results about rigorously computing solutions of square nonlinear systems of equations. Additionally, we summarize basic definitions and known results about slope forms, which will be an important tool when extending the exlusion region-concept from \cite{schichl2005ex} to the parameter-dependent case. In Section \ref{sec:fixs0} we will outline the method introduced by \citet{schichl2005ex} as it is the starting point for the new method. In Section \ref{sec:paramProb} we will then state and prove the main results of this paper and describe how to extend the inclusion/exclusion-region concept to the parameter-dependent case. In Section \ref{sec:examples} the new method is demonstrated on a numerical example. Finally, in Section \ref{sec:discussion} a summary as well as an outlook of future work is given.
\paragraph{Notation.}
Throughout the paper, we will use the following notation: For a matrix $A\in \zz{R}^{n\times m}$ we denote by $A_{: K}$ the $n\times k$ submatrix consisting of $k$ columns with indices in $K\subseteq \lbrace{1,\dots,m}\rbrace$, and, similarly $A_{K:}$ denotes the $k\times m$ submatrix with row-indices in $K\subseteq \lbrace 1,\dots,n\rbrace$. Let $F\colon \zz{R}^m\to\zz{R}^n$. If $y = (x,s)^T\in \zz{R}^m$ is a partition of $x$ with $x=y_I$, $s=y_J$, where $I$, $J$ are index sets with $I\cap J=\emptyset$ and $I\cup J=\lbrace{1,\dots,m}\rbrace$, the Jacobian of $F$ with respect to $x$ is
  \begin{equation*}
    F'_x(y)= \frac{\partial F}{\partial x}(y) = \left(\frac{\partial F}{\partial y}(y)\right)_{:I}.
  \end{equation*}
 A slope of $F$ with center $z$ at $y$ is written as $\slope[z,y]{F}$, a slope with respect to $x$ then is
 \begin{equation*}
  \slope[z,y]{F_x} = (\slope[z,y]{F})_{:I}.
 \end{equation*}
    Since second order slopes (resp.~first order slopes of the Jacobian) are third order tensors, we use the following multiplication rules (see \cite{schichl2014exOpt}) for a 3--tensor $\mathcal{T} \in \zz{R}^{n \times m\times r}$, a vector $v \in \zz{R}^r$, and matrices $C\in \zz{R}^{s\times n}$, $B\in \zz{R}^{r \times s}$:
   \begin{alignat}{3}\label{eq:tensorRules}
     \left(\mathcal{T}v\right)_{ij}  = \sum_{k=1}^r \mathcal{T}_{ijk}v_k,  &&\quad\quad
     \left(C\mathcal{T} \right)_{ijk} = \sum_{l=1}^n C_{il}\mathcal{T}_{ljk}, &&\quad\quad
     \left(\mathcal{T}B\right)_{ijk}  = \sum_{l=1}^r \mathcal{T}_{ijl}B_{lk}.
   \end{alignat}
   
   Additionally, we define for a vector $v\in\zz{R}^n$ and a 3--tensor $\mathcal{T}\in \zz{R}^{n\times n\times n}$ the product
   \begin{equation*}
    v^T\mathcal{T}\;v=(\mathcal{T}v)\; v,  \quad \text{i.e., with }\quad  \left(v^T\mathcal{T}\;v\right)_i = \sum_{j=1}^n \sum_{k=1}^n v_k \mathcal{T}_{ijk}\,v_j
   \end{equation*}
 
\section{Preliminaries and known results} \label{sec:slopeForms}
   Consider a twice continuously differentiable function $F\colon D \subseteq \zz{R}^n \to \zz{R}^n$. We can always write (see \cite{schichl2005ex})
   \begin{equation}\label{eq:slope}
     F(x) - F(z) = \slope[z,x]{F}\;(x-z),
   \end{equation}
   for any two points $z$, $x\in D$ and a suitable matrix $\slope[z,x]{F}\in \zz{R}^{n\times n}$, a so-called \fdef{slope matrix} for $F$ with center $z$ at $x$. While in the multivariate case, the slope matrix is not uniquely determined, we always have by differentiability
   \begin{equation*}
     \slope[z,z]{F} = F'(z).
   \end{equation*}
   Assuming that the slope matrix is continuously differentiable in both points, we can write similarly
   \begin{align}
     \slope[z, x]{F} &= \slope[z,z']{F}+\slope[z, z', x]{F}(x_k-z'_k)\label{eq:FslopeEst1}\\
     \intertext{which simplifies for $z=z'$ to}
     \slope[z, x]{F} &=F'(z) + \slope[z, z, x]{F}(x_k-z_k),\label{eq:FslopeEst2}
   \end{align}
   where the \fdef{second order slopes} $\slope[z,z',x]{F}$, $\slope[z,z,x]{F}$,  respectively, are continuous in $z$, $z'$ and $x$. If $F$ is quadratic, the first order slopes are linear, and thus, the second order slope matrices are constant. 
   Let $z$ be a fixed center in the domain of $F$. Having a slope $\slope[z,x]{F}$ for all $x \in \intval{x}$ we get
   \begin{equation}\label{eq:slopeInclusion}
     F(\intval{x}) \subseteq F(z) + \slope[z,\intval{x}]{F}\left(\intval{x}- z\right),
   \end{equation}
   and, analogously,
   \begin{equation}\label{eq:slopeInclusion2}
     F(\intval{x}) \subseteq F(z) + \left(F'(z) + \slope[z, z, \intval{x}]{F}(\intval{x}_k-z_k)\right) \left(\intval{x}- z\right).
   \end{equation}
    Hence, the \fdef{first} and \fdef{second order slope forms} given in \eqref{eq:slopeInclusion} and \eqref{eq:slopeInclusion2}, respectively, provide enclosures for the true range of the function $F$ over an interval $\intval{x}$.  There are recursive procedures to calculate slopes, given $x$ and $z$ (see \cite{krawczyk1985interval, rump1996, kolev1997, schichl2005interval}). A \falg{Matlab} implementation for first order slopes is in \falg{Intlab} \cite{rump1999}; also, the \falg{Coconut} environement \cite{schichl2012algorithmic} provides algorithms.
   
   Similarly to derivatives, slopes obey a sort of chain rule. Let $F \colon \zz{R}^m\to \zz{R}^p$, $g\colon \zz{R}^n\to \zz{R}^m$. Then we have
   \begin{align}\label{eq:chainRule}
     \begin{split}
       (F\circ g)(x) & = (F\circ g)(z) + \slope[z,x]{(F\circ g)}(x-z)\\
       & = F(g(z)) + \slope[g(z), g(x)]{F}\; \slope[z,x]{g}(x-z),
     \end{split}
   \end{align}
i.e., $\slope[g(z),\,g(x)]{F}\;\slope[z,x]{g}$ is a slope matrix for $F\circ g$.

 Exclusion regions for $n\times n$-systems are usually constructed using uniqueness tests based on the Krawczyk operator (see \cite{neumaier1991}) or the Kantorovich theorem (see \cite{kantorovich1952, ortega2000, deuflhard1979}), which both provide existence and uniqueness regions for zeros of square systems of equations. \citet{kahan1968} used the Krawczyk operator to make existence statements. An important advantage of the Krawczyk operator is that it only needs first order information. Together with later improvements about slopes, his result is contained in the following statement.

   \begin{theorem}[Kahan]\label{thm:kraw}
     Let $F\colon \zz{R}^n\to\zz{R}^n$ be as before and let $z\in\intval{z}\subseteq\intval{x}$. If there is a matrix $C\in\zz{R}^{n\times n}$ such that the Krawczyk operator
     \begin{equation}\label{eq:krawall}
       \mathrm{K}({\intval{z}, \intval{x}}) \assign z-CF(z)-\left(\slope[\intval{z}, \intval{x}]{CF}-\id\right)(\intval{x}-z)
     \end{equation}
     satisfies $\mathrm{K}({\intval{z}, \intval{x}}) \subseteq \intval{x}$, then $\intval{x}$ contains a zero of $F$. Moreover, if $\mathrm{K}({\intval{x}, \intval{x}})\subseteq \interior (\intval{x})$, then $\intval{x}$ contains a unique zero. 
   \end{theorem}
   \citet{neumaier1990} proved that the Krawczyk operator with slopes always provides existence regions which are at least as large as those computed by Kantorovich's theorem. Based on a more detailed analysis of the properties of the Krawczyk operator, \citet{schichl2005ex} provided componentwise and affine invariant existence, uniqueness, and nonexistence regions given a zero or any other point in the search region. More recently, this concept was extended to optimiziation problems; see \citet{schichl2014exOpt}. 

\section{Inclusion/exclusion regions for a fixed parameter}\label{sec:fixs0}
  We consider the nonlinear system of equations \eqref{eq:H} at a fixed parameter value $p$,
   \begin{equation}\label{eq:Ht}
     H(x, p) = 0, \quad H\colon X \subseteq \zz{R}^n \to \zz{R}^n.
   \end{equation}
   Let $z\approx x(p)$ be an approximate solution of \eqref{eq:Ht}, i.e., 
   \begin{equation}\label{eq:approxSol}
    H(z,p) \approx 0.
   \end{equation}
   Our first aim is the verification of a true solution $x^*$of system \eqref{eq:Ht} in a neighbourhood of $z$ by computing an inclusion (resp.~exclusion) region around $z$ as described by \citet{schichl2005ex}. Assuming regularity of the Jacobian $H'_x(z, p)$ we take
   \begin{equation}\label{eq:C}
     C \approx H'_x(z, p)^{-1} \in \zz{R}^n \times \zz{R}^n
   \end{equation}
   as a fixed preconditioning matrix and compute the componentwise bounds 
   \begin{align}\label{eq:bounds0}
     \begin{split}
       \ubar{b} &\geq \absv{CH(z, p)}\\
       B_0&\geq \absv{CH'_x(z, p)-\id}\\
       B(x) &\geq \absv{C\slope[(z, p), (z, p), (x, p)]{H_{xx}}} \\
       \ubar{B}&\geq B(x) \quad \forall \; x \in\; \mathbf{x}\subseteq X,
     \end{split}
   \end{align}
    where the second order slopes $H_{xx}$ are fixed. Throughout the paper, we assume $z \approx x(p)\in \intval{x}$ as fixed center and the bounds from \eqref{eq:bounds0} valid for all $x\in\intval{x}$, where $\intval{x}\subseteq X$ is chosen appropriately (see below). Following \cite[Thm.~4.3]{schichl2005ex}, we choose a suitable vector $0<v\in\zz{R}^n$, which basically determines the scaling of the inclusion/exclusion regions, and set
   \begin{equation}\label{eq:aw}
     w \assign \left(\id -B_0\right)v, \quad a \assign v^T\ubar{B}\;v
   \end{equation}
   Supposing 
   \begin{equation*}
     D_j = w_j^2-4a_j\ubar{b}_j>0
   \end{equation*}
   for all $j=1,\dots,n$, we define 
   \begin{align}\label{eq:lambdaExIn}
     \lambda_j^e&\assign \frac{w_j+\sqrt{D_j}}{2a_j}, &\quad \lambda_j^i &\assign \frac{\ubar{b}_j}{a_j\lambda_j^e},\\ 
     \lambda^e &\assign \min_{j=1,\dots,n}\lambda^{e}_j,&\quad \lambda^i &\assign \max_{j=1,\dots, n} \lambda^{i}_j.\nonumber
     \end{align}
   If $\lambda^e>\lambda^i$, then there is at least one zero $x^*$ of \eqref{eq:Ht} in the inclusion region $\intval{R}_{i}^0$ and the zeros in this region are the only zeros of \eqref{eq:Ht} in the interior of the exclusion region $\intval{R}_{e}^0$ with 
   \begin{align}\label{eq:exR0}
     \begin{split}
       \intval{R}^{i}_0 & \assign [z\!-\lambda^i\,v,\; z\!+\lambda^i\,v]\subseteq \intval{x}\\
       \intval{R}^{e}_0 & \assign [z\!-\lambda^e\,v,\; z\!+\lambda^e\,v]\cap \intval{x}.
     \end{split}
   \end{align}
   In the important special case where $H(x,p)$ is quadratic in $x$, the first order slope $\slope[(z, p), (x,p)]{H}$ is linear in $x$. Hence, all second order slope matrices are constant in $x$. Therefore, the upper bounds $B(x)=B$ are constant as well. Thus, we can set $\ubar{B} = B$ and the estimate from \eqref{eq:bounds0} becomes valid everywhere. Otherwise, an appropriate choice of $\intval{x}\subseteq X$ is crucial in order to keep the bounds $\ubar{B}$ on the second order slopes considerably small.

\section{Parameter-dependent problem}\label{sec:paramProb}
  Let $(z,p)$ be an approximate solution of \eqref{eq:H} for which a pair of inclusion and exclusion regions can be computed as described in Section \ref{sec:fixs0}. In addition, we assume the bounds from \eqref{eq:bounds0} valid for $\intval{x}\subseteq X$ with $z\in \intval{x}$. We aim to prove the existence of a solution of \eqref{eq:H} for every $s\in \intval{s}\subseteq S$. Therefore, we first extend the results from \citet{schichl2005ex} to the parameter-dependent case. In Thm.~\ref{thm:compMu} we then state a method to explicitly construct such a parameter interval $\intval{s}$. As a by-product we get an outer enclosure of a solution region $\intval{x}(\intval{s})\subseteq \intval{x}$ over the parameter set $\intval{s}$. 
  
 Consider any box $\intval{s}\subseteq S\subseteq \zz{R}^p$ with $p\in \intval{s}$ and a continuously differentiable approximation function 
  \begin{equation}\label{eq:that}
    \hat{x}\colon \zz{R}^p\to\zz{R}^n
  \end{equation}
   which satisfies $\that(p) = z$ and $\that(s)\in\intval{x}$ for all $s\in \intval{s}$, and prove for every $s\in \intval{s}$ the existence of an inclusion box  $$\mathbf{R}^i_s\subseteq \mathbf{x}\quad \text{ with }\quad 0\in H\left(\mathbf{R}^i_s,s\right).$$
 Note that the choice of the approximation function may greatly influence the quality, i.e., the radius, of the parameter interval $\intval{s}$ (see Section \ref{sec:examples}).
  
  We define 
   \begin{equation}\label{eq:g}
     g(s) = \begin{pmatrix}\that(s)\\ s\end{pmatrix}, \quad g\colon \intval{s}\to X \times S \subseteq \zz{R}^{n}\times \zz{R}^{p}.
   \end{equation}
   With $\slope[p, s]{\that}$ denoting a slope matrix for $\that$ with center $p$ at $s$, a slope matrix for $g$ is given by
  \begin{equation}\label{eq:gslope}
    \slope[p, s]{g} = \begin{pmatrix}\slope[p, s]{\that}\\ \id\end{pmatrix} \quad \in \zz{R}^{(n+p)\times p},
  \end{equation}
  since
  \begin{equation*}
    g(s)-g(p) = \begin{pmatrix}
      \that(s)-\that(p)\\
      s-p
    \end{pmatrix} = \begin{pmatrix}
      \slope[p, s]{\that}\\
      \id
    \end{pmatrix} \left(s-p\right).
  \end{equation*}

 Let $C$ be the fixed preconditioning matrix from \eqref{eq:C}. For each $s \in \intval{s}$ we define similar bounds as in \eqref{eq:bounds0}
  \begin{align}
      \ubar{\mathfrak{b}}(s) &\geq \absv{CH(\that(s),s)},\label{eq:boundsNew_b}\\ 
      \mathfrak{B}_0(s)&\geq \absv{CH'_x(\that(s),s)-\id},\label{eq:boundsNew_B}\\
      \mathfrak{B}(x,s) &\geq \absv{C\slope[ (\that(s),s), (\that(s),s), (x,s)]{H_{xx}}} \label{eq:boundsNew_Bj}
  \end{align}
  and calculate estimates on the bounds from \eqref{eq:boundsNew_b} and \eqref{eq:boundsNew_B} with respect to the bounds from \eqref{eq:bounds0} using first order slope approximations. Applying the chain rule \eqref{eq:chainRule} to $H(\that(s),s) = (H\circ g)(s)$ we get 
  \begin{equation}\label{eq:estHslope}
    H(\that(s),s) = H(g(p)) + \slope[g(p), g(s)]{H}\,\slope[p,\;s]{g}\,(s-p),
  \end{equation}
and, similarly we estimate the first derivative of $H$ with respect to $x$ by
\begin{align*} \label{eq:estHprimeSlope}
    H'_x(\that(s),s)
    & = H'_x(z, p)+
    \slope[g(p),\; g (s)]{(H'_x)}\, \slope[p,\;s]{g}\,(s-p),
  \end{align*}
where the 3--tensor $\slope[g(p),\; g (s)]{(H'_x)} \in \zz{R}^{n\times n \times (n+p)}$ is a slope for $H'_x$.

By taking absolute values we get with $\til{y} \assign \absv{s-p}$ and \eqref{eq:bounds0}
 \begin{equation*}
  \begin{array}{lcccl}
    \absv{CH(\that(s),s)}
    & \leq &  \ubar{b}  & + & \absv{\slope[g(p),\; g(s)]{CH}}\, \absv{\slope[p,\; s]{g}}\, \til{y},\\
  \absv{CH'_x(\that(s),s)-\id} 
    & \leq & B_0        & + & \absv{C\slope[g(p),\; g(s)]{((H'_x))}}\, \absv{\slope[p,\; s]{g}}\, \til{y}.
  \end{array}
 \end{equation*}
 Hence, we define
    \begin{equation}\label{eq:obarbnew}
    \ubar{\mathfrak{b}}(s)\assign \ubar{b} + G_0(s)\, \til{y} \quad \text{ with }\quad G_0(s)\assign \absv{\slope[g(p),\; g(s)]{CH}}\, \absv{\slope[p,\; s]{g}}
   \end{equation}
and
  \begin{equation}\label{eq:BzeroNew}
  \mathfrak{B}_0(s)\assign B_0  + A(s)\, \til{y}\quad \text{ with } \quad A(s) \assign \absv{C\slope[g(p),\; g(s)]{(H'_x)}}\, \absv{\slope[p,\; s]{g}}.
   \end{equation}
   Note, that $A(s) \in \zz{R}^{n\times n\times p}$ is the result of the multiplication of a 3--tensor with an $((n+p)\times p)$-matrix. Therefore, $A(s)$ is computed by the appropriate multiplication rule from \eqref{eq:tensorRules}.

   \begin{prop}\label{prop1}
	 Let $(z,p)\in \intval{x}\times \intval{s}\subseteq X\times S$, where $\intval{x}$, $\intval{s}$ are any subboxes of $X$ and $S$ containing $(z,p)$ (from \eqref{eq:approxSol}) such that the bounds \eqref{eq:bounds0} hold for all $x\in \intval{x}$. Additionally, let $s\in\intval{s}$ be an arbitrary parameter value, and $\that \assign \that(s)\in \interior(\intval{x})$ be the function value of the approximation function from \eqref{eq:that} at $s$. Further, let $0<v\in\zz{R}^n$ and $\lambda^e$ as in \eqref{eq:lambdaExIn}. Then for a true solution $x=x(s)$ of \eqref{eq:H} at $s$ with $\absv{x-z}\leq \lambda^e\,v$ the deviation
	 $$d_s\assign \absv{x-\that}$$ satifies
	     \begin{equation}\label{eq:devEst}
	       0\leq d_s \leq \ubar{\mathfrak{b}}(s) + \left(\mathfrak{B}_0(s)+\mathfrak{B}(x,s)\,d_s\right)d_{s}
	     \end{equation}
	     with $\ubar{\mathfrak{b}}(s)$, $\mathfrak{B}_0(s)$, and $\mathfrak{B}(x,s)$ as defined in \eqref{eq:obarbnew}, \eqref{eq:BzeroNew}, and \eqref{eq:boundsNew_Bj}, respectively.
   \end{prop}
   \begin{proof} 
     Let $(x^1, s^1)$ be an arbitrary point in the domain of definition of $H$. Then we have by \eqref{eq:slope}
     \begin{equation*}
	     H(x, s) = H(x^1, s^1) + \slope[(x^1,\; s^1), (x,s)^T]{H} \, \begin{pmatrix}                                                           
	      x-x^1\\                                                                                                    s-s^1
	     \end{pmatrix} = 0
     \end{equation*}
     since $x$ is a solution of \eqref{eq:H} at $s$. This simplifies for  $(x^1, s^1)= (\that, s)$ and $g(s)$ as in \eqref{eq:g} to
     \begin{equation}\label{eq:prop41a}
       H(x, s)  = H(g(s)) + \slope[g(s),\; (x,s)^T]{H_{x}} \, (x-\that),                                                                                                            
     \end{equation}
     where we calculate $H(g(s))$ by \eqref{eq:estHslope} with respect to $(z, p)$ as 
     \begin{equation}\label{eq:prop41b}
      H(g(s))  =
      H(g(p))+  %
       \slope[p, s]{(H\circ g)}
       \, ( s-p)
     \end{equation}
	with $  \slope[p, s]{(H\circ g)}\assign \slope[g(p), g(s)]{H}\, \slope[p,s]{g}$, and $\slope[g(s), (x,s)^T]{H_x}$ by \eqref{eq:FslopeEst2} as 
     \begin{equation}\label{eq:prop41c}
       \slope[g(s), (x,s)^T]{H_{x}} 
        = 
       H'_x(g(s)) + \slope[g(s),\; g(s),\; (x,s)^T]{H_{xx}}\,(x-\hat{x})
     \end{equation}
     with $g(s)$, $\slope[p, s]{g}$ as in \eqref{eq:g} and \eqref{eq:gslope}.
     
     Now we consider the deviation between the approximate and a true solution and get with \eqref{eq:prop41a}
     \begin{alignat*}{2}
       -(x-\that) & =  -(x-\that) 
       && + CH(g(s)) + \slope[g(s),\;(x,s)^T]{H_{x}}(x-\that)
     \intertext{which extends by \eqref{eq:prop41c} to}
       & = C H(g(s)) && +(CH'_x(g(s))-\id ) (x-\that)\\
       & && +(x-\that)^T\,\slope[g(s), g(s), (x,s)^T]{CH_{xx}}\, (x-\that).
       \intertext{Taking absolute values, we get by \eqref{eq:boundsNew_Bj}, \eqref{eq:obarbnew}, and  \eqref{eq:BzeroNew}} 
       d_s = \absv{x-\that} & \leq \absv{C H(g(s))} && 
       + \absv{CH'_x(g(s))-\id}  \absv{x-\that}\\
       & && +\absv{x-\that}^T\,\absv{\slope[g(s),\; g(s),\; (x,s)^T]{CH_{xx}}}\,  \absv{x-\that}\\
       & \leq \quad \ubar{\mathfrak{b}}(s) && + \left(\mathfrak{B}_0(s)+\mathfrak{B}(x,s)\,d_s\right)\,d_{s}.
     \end{alignat*}
   \end{proof}%
   
   Using this result, we are able to formulate a first criterion for existence regions.
   
  \begin{theorem}\label{thm:pexReg}
   Let again $s\in\intval{s}$ with corresponding function value $\that \assign \that(s)\in \interior(\intval{x})$ of the approximation function from \eqref{eq:that} at $s$. In addition to the assumptions from Prop.~\ref{prop1} let $0<u\in \zz{R}^{n}$ be such that
    \begin{equation}\label{eq:ineqU}
      \ubar{\mathfrak{b}}(s) + \left(\mathfrak{B}_0(s)+\mathfrak{B}(s)\,u\right)\,u\leq u
    \end{equation}
    with $\mathfrak{B}(s)\geq \, \mathfrak{B}(x,s)$ for all $x$ in $M_u(s)$, where
    \begin{equation}\label{eq:Mus}
      M_u(s) \assign \left\lbrace x \mid \absv{x-\that}\leq u\right\rbrace\subseteq \intval{x}.
    \end{equation}
    Then \eqref{eq:H} has a solution $x(s) \in M_u(s)$. 
  \end{theorem}  
  \begin{proof} 
    For arbitrary $x$ in the domain of definition of $H$ we define
    \begin{equation*}
    \mathrm{K}_s(x) \assign x - CH(x,s).
    \end{equation*}
    For $x\in M_u(s)$ we get with \eqref{eq:prop41a} and \eqref{eq:prop41c}
    \begin{alignat*}{2}
      \mathrm{K}_s(x) & = \that - \Big( CH(g(s)) &&+ (CH'_x (g(s)) -\id ) (x-\that )\\
      &  && +   (x-\that )^T\,\slope[g(s),\: g(s),  (x,s )^T]{CH_{xx}}\, (x-\that )\Big).
    \end{alignat*}
    Taking absolute values we get
    \begin{alignat}{2}\label{eq:estKraw}
      \absv{\mathrm{K}_s(x) -\that} &\leq \absv{CH (g(s) )} && + \absv{C H'_x(g(s)) -\id}\absv{x-\that}\\
      & && + \absv{x-\that}\,\absv{\slope[g(s),\: g(s),  (x,s)^T]{CH_{xx}}}\,\absv{x-\that}\nonumber\\
      & \leq \quad\quad \ubar{\mathfrak{b}}(s)  && +  \left(\mathfrak{B}_0(s) + \mathfrak{B}(s)\,u \right)\,u\nonumber\\
      &\leq \quad\quad u \nonumber
    \end{alignat}
    by assumption \eqref{eq:ineqU}. Thus, $\mathrm{K}_s(x) \in M_u(s)$ for all $x \in M_u(s)$, and since $\mathrm{K}_s(x)$ is equal to the Krawczyk operator \eqref{eq:krawall} for a fixed parameter $s$, we get by Prop.~\ref{thm:kraw} that there exists a solution of \eqref{eq:H} in $M_u(s)$.
  \end{proof}

  Based on the above results, the following theorem provides a way of constructing inclusion and exclusion regions for an approximate solution $\that(s)$. 
    \begin{theorem} \label{thm:pinPex}  
      In addition to the assumptions from Prop.~\ref{prop1} and Thm.~\ref{thm:pexReg}, we take 
      \begin{equation*}
	\mathfrak{B}(s) \geq \mathfrak{B}(x,s) \quad \forall \; x \in \intval{x}.
      \end{equation*}
      For $0<v\in \zz{R}^n$ we define
      \begin{equation*}
      \mathfrak{w}(s) \assign \left(\id -\mathfrak{B}_0(s)\right)v, \quad \mathfrak{a}(s) = v^T\, \mathfrak{B}(s)\,v.
      \end{equation*}
      If
      \begin{equation}\label{eq:discT}
	D_j(s) \assign \mathfrak{w}_j(s)^2 - 4\,\mathfrak{a}_j(s)\,\ubar{\mathfrak{b}}_j(s) > 0 
      \end{equation}
      for all $j=1,\dots,n$, we define 
      \begin{align}\label{eq:lambdajS}
	\lambda_j^e(s) \assign\frac{\mathfrak{w}_j(s)+\sqrt{D_j(s)}}{2\mathfrak{a}_j(s)}, \quad &\lambda_j^i(s) \assign \frac{\ubar{\mathfrak{b}}_j(s)}{\mathfrak{a}_j(s)\cdot\lambda_j^e(s)}     
      \end{align}
      and
      \begin{align}\label{eq:lambdaS}
	\lambda_s^e  \assign \min_{j} \; \lambda_j^e(s), \quad &\lambda_s^i \assign \max_{j}\;\lambda_j^i(s).
      \end{align}
      If $\lambda_s^e > \lambda_s^i$ and  
      \begin{equation}\label{eq:boundary}
      \left(\that_j +\left[-1,\, 1\right]\,\lambda_s^i\, v_j\right) \subseteq \intval{x}_{j} \quad \text{ for all }j,
      \end{equation}
      then there exists at least one zero $x^*$  of \eqref{eq:H} for a parameter set $s$ (i.e., with $H(x^*,s)=0$) in the \fdef{inclusion region}
      \begin{equation}\label{eq:inclusionRegS}
      \mathbf{R}^i_s \assign \left[\that-\lambda_s^iv,\; \that+\lambda_s^iv\right]\subseteq \mathbf{x},
      \end{equation} 
      and these zeros are the only zeros of $H$ at $s$ in the interior of the \fdef{exclusion region}
      \begin{equation*}
      \mathbf{R}^e_s \assign \left[\that-\lambda_s^ev,\; \that+\lambda_s^ev\right]\cap \mathbf{x}.
      \end{equation*}
    \end{theorem}
    \begin{proof}
      We set $u=\lambda v$ with arbitrary $0<v\in\zz{R}^{n}$, and check for which $\lambda=\lambda(s) \in\zz{R}_+$ the vector $u$ satisfies property \eqref{eq:ineqU}. We get
      \begin{align*}
	\lambda v &\geq \ubar{\mathfrak{b}}(s) + \left(\mathfrak{B}_0(s) + \lambda\, \mathfrak{B}(s)\,v\right)\,\lambda v\\
	& = \ubar{\mathfrak{b}}(s) + \lambda \left(v-\mathfrak{w}(s)\right) + \lambda^2 \mathfrak{a}(s),
      \end{align*}
      which leads to the sufficient condition
      \begin{equation*}
	\lambda^2\mathfrak{a}(s)-\lambda \mathfrak{w}(s) + \ubar{\mathfrak{b}}(s) \leq 0.
      \end{equation*}
      The $j$th component of this inequality requires $\lambda$ to be between the solutions of the quadratic equation
      \begin{equation*}
	\lambda^2\mathfrak{a}_j(s)-\lambda \mathfrak{w}(s)_j + \ubar{\mathfrak{b}}(s)_j= 0,
      \end{equation*}
      which are exactly $\lambda_j^i(s)$ and $\lambda_j^e(s)$. Since $D_j(s)>0$ for all $j$ by assumption, the interval $\left[\lambda_s^i,\; \lambda_s^e\right]$ is nonempty. Thus, for all $\lambda(s) \in \left[\lambda_s^i,\; \lambda_s^e\right]$, the vector $u$ satisfies \eqref{eq:ineqU}. 
      
      It remains to check, whether the solution(s) in $\mathbf{R}^i_s$ are the only ones in $\mathbf{R}^e_s$. Assume that $x$ is a solution of \eqref{eq:H} at $s$ with $x \in \interior(\mathbf{R}^e_s)\setminus \mathbf{R}^i_s$, and let $\lambda=\lambda(s)$ be minimal with $\absv{x-\that} \leq \lambda v$. By construction, we have $\lambda_s^i < \lambda < \lambda_s^e$. In the proof of Thm.~\ref{thm:pexReg} we got for the Krawczyk operator \eqref{eq:krawall} 
      \begin{equation*}
      \mathrm{K}_s(x) \assign x-CH(x,s) = x,
      \end{equation*}
      since $x$ is a solution of \eqref{eq:H} at $s$. Thus, we get by the same considerations as in the proof of Thm.~\ref{thm:pexReg} from \eqref{eq:estKraw}
      \begin{equation*}
	\absv{x-\that}\leq\ubar{\mathfrak{b}}(s) + \left(\mathfrak{B}_0(s) + \lambda\,\mathfrak{B}(s)\,v\right)\,\lambda v < \lambda v,
      \end{equation*}
      since $\lambda > \lambda_s^i$. Since this contradicts the minimality of $\lambda$, there is no solution of \eqref{eq:H} at $s$ in $\interior (\mathbf{R}^e_s)\setminus \mathbf{R}^i_s$.  So, if \eqref{eq:boundary} is satisfied for all $j$, there exists at least one solution $x^*$ of \eqref{eq:H} at $s$ in the inclusion box $\mathbf{R}^i_s$ and there are no other solutions in $\interior(\mathbf{R}^e_s)\setminus\mathbf{R}^i_s$.
    \end{proof}
The final step is now to compute a feasible parameter $0<\mu\in\zz{R}$ such that Thm.~\ref{thm:pinPex} holds for all 
\begin{equation*}
 s\in \hat{\intval{s}}\assign\left[p-\mu\,y, \; p+\mu\,y\right],
\end{equation*}
with arbitrary scaling vector $y\in \zz{R}^p$. Assume $\widetilde{\intval{s}}\subseteq \intval{s}\in S$, where $\intval{s}$ is an arbitrary box containing p. We compute a lower bound on each component 
  \begin{equation*}
    D_j(s) \assign \mathfrak{w}_j(s)^2 - 4\mathfrak{a}_j(s)\:\ubar{\mathfrak{b}}_j(s)
  \end{equation*}
from the positivity requirement \eqref{eq:discT} of Thm.~\ref{thm:pinPex} over the box $\intval{s}$. For 
the bounds from \eqref{eq:obarbnew} and \eqref{eq:BzeroNew} we compute upper bounds
  \begin{align}
  \ubar{\mathfrak{b}} &\assign \ubar{\mathfrak{b}(\intval{s})} = \ubar{b} + \mu\: \ubar{G_0}\: y \quad \text{ with }\quad \ubar{G_0} \assign \ubar{G_0(\intval{s})},  \label{eq:upobarb} \\
  \ubar{\mathfrak{B}_0} &\assign \ubar{\mathfrak{B}_0(\intval{s})} = B_0+\mu\:\ubar{A}\:y \quad \text{ with } \quad \ubar{A}\assign\ubar{A(\intval{s})}.\label{eq:upBzero}
  \end{align}
  Since by construction $\ubar{G_0}\geq G_0(s)$ and $\ubar{A}\geq A(s)$ for all $s\in \intval{s}$, we have
  \begin{equation*}
  \ubar{\mathfrak{b}}\geq\ubar{\mathfrak{b}}(s)\quad \text{ and }\quad \lbar{\mathfrak{w}} \assign \left(\id-\ubar{\mathfrak{B}_0}\right)v \leq \left(\id-\mathfrak{B}_0(s)\right)v=\mathfrak{w}(s)
  \end{equation*}
  for all $s\in \intval{s}$. By computing an upper bound over an appropriate box $\intval{x}\in X$ with $z\in \intval{x}$ (e.g., take $\intval{x} = k\, \intval{R}_0^i$ with $k\in \zz{R}_+,\, k \leq 1$) we get upper bounds on the second order slopes from \eqref{eq:boundsNew_Bj}
    \begin{equation}\label{eq:upBj}
      \ubar{\mathfrak{B}} \assign \ubar{\absv{\slope[g(\intval{s}), g(\intval{s}), (\mathbf{x}, \intval{s})^T ]{CH_{xx}}}},
    \end{equation}
  which satisfy
  \begin{equation*}
   \ubar{\mathfrak{B}}_j \geq \mathfrak{B}_j(s) \geq \mathfrak{B}_j(x,s) \quad \text{ for all } s \in \intval{s},\; x\in \intval{x}.
  \end{equation*}

 Hence, the lowest values of the discriminant $D$ from \eqref{eq:discT} are obtained by 
    \begin{equation*}
    \lbar{D}_j = \lbar{\mathfrak{w}}_j^2-4\;\ubar{\mathfrak{a}}_j\;\ubar{\mathfrak{b}}_j
  \end{equation*}
 with $\ubar{\mathfrak{b}}$ from \eqref{eq:upobarb} and 
  \begin{equation}\label{eq:param4Dj}
      \lbar{\mathfrak{w}} \assign \mathfrak{w}(\mu)= \left(\id-\ubar{\mathfrak{B}_0}\right)v,\qquad \ubar{\mathfrak{a}} \assign v^T\, \ubar{\mathfrak{B}}\, v.
      \end{equation}

  Considering $\lbar{D}_j = \lbar{D}_j(\mu)$, we get
  \begin{equation}\label{eq:Djmu}
    \lbar{D}_j(\mu) = \alpha_j^2\:\mu^2-2\:\beta_j\:\mu + \gamma_j
  \end{equation}
 with 
  \begin{equation}\label{eq:alphaBeta}
  \alpha_j \assign \left((\ubar{A}y)v\right)_j, \qquad
  \beta_j \assign \alpha_j w_j+2\:\ubar{\mathfrak{a}}_j\left(\ubar{G_0}y\right)_j, \qquad
  \gamma_j = w_j^2-4\:\ubar{\mathfrak{a}}_j\:\ubar{b}_j
  \end{equation}
  and $w_j$ and $\ubar{b}_j$ from \eqref{eq:bounds0}. Solving each quadratic equation 
  \begin{equation}\label{eq:quadGlDmu}
    \lbar{D}_j(\mu) = \alpha_j^2\:\mu^2-2\:\beta_j\:\mu + \gamma_j = 0
  \end{equation}
  for $\mu$, we get 
  \begin{equation}\label{eq:mushort}
    \mu_j = \frac{\beta_j \pm \sqrt{\beta_j^2-\alpha_j^2\:\gamma_j}}{\alpha_j^2}.
  \end{equation}
  Since $H(z,p)\approx 0$, the discriminant in \eqref{eq:mushort} is smaller than $\beta_j^2$, since we have
  \begin{equation*}
    \beta_j^2-\alpha^2\gamma_j = \beta_j^2-\alpha_j^2(w_j^2-4\:\ubar{\mathfrak{a}}_j\:\ubar{b}_j)
  \end{equation*}
   with $\ubar{b}_j$ being an upper bound for the function value at $(z, p)$ and thus, close to zero. Hence, both solutions of \eqref{eq:quadGlDmu} are positive. In order to derive numerically stable results, we compute these solutions by
    \begin{equation}\label{eq:mustable}
      \ubar{\mu}_j =  \frac{\beta_j + \sqrt{\beta_j^2-\alpha_j^2\gamma_j}}{\alpha_j^2}, \qquad \lbar{\mu}_j=\frac{\gamma_j}{\alpha_j^2\:\ubar{\mu}_j}.
    \end{equation}
    and set 
    \begin{equation*}
      \mu \assign \min_j \; \lbar{\mu}_j,
    \end{equation*}
    since we need $\mu \in [0,\, \lbar{\mu}_j]$ in order to meet the positivity requirement \eqref{eq:discT} in the $j$-th component.
    
    Now we are able to state and prove the main result of this section.
 
  \begin{theorem}\label{thm:compMu}
  Let $\intval{s}\in S$ with $p\in \intval{s}$, $\intval{x}\in X$ with $z\in \intval{x}$ as above. In addition to the assumptions of Thm.~\ref{thm:pinPex} we assume upper bounds
    \begin{enumerate}
      \item on the first order slope of $H(g(s))$,
     \begin{equation}\label{eq:G0thm}
      \ubar{G}_0  \assign \ubar{G_0(\intval{s})} = \ubar{\absv{\slope[g\left(p\right), g(\intval{s})]{CH}}\, \absv{\slope[p,\; \intval{s}]{g}}} ,
    \end{equation}
    \item  on the slope of the first derivative of $H(g(s))$ wrt.~$x$, 
    \begin{equation}\label{eq:Athm}
      \ubar{A} \assign \ubar{A(\intval{s})} = \ubar{\left\vert C\slope[g(p),\; g(\intval{s})]{(H'_x)}\right\vert\, \left\vert\slope[p,\; \intval{s}]{g}\right\vert},
    \end{equation}
    \item and on the second order slopes of $H(g(s))$ wrt.~$x$
    \begin{equation*}
     \ubar{\mathfrak{B}}\assign \ubar{\absv{C\slope[g(\intval{s}),\; g(\intval{s}),\; (\intval{x},\intval{s})^T]{H_{xx}}}}
    \end{equation*}
    \end{enumerate}
    which hold for all $s\in\intval{s}$, $x\in \intval{x}$.  
    Let further $0<y\in\zz{R}^{p}$, $0<v\in \zz{R}^{n}$ as before,
    \begin{equation}\label{eq:lbarMuj}
    \ubar{\mu}_j =  \frac{\beta_j + \sqrt{\beta_j^2-\alpha_j^2\gamma_j}}{\alpha_j^2}, \qquad \lbar{\mu}_j=\frac{\gamma_j}{\alpha_j^2\:\ubar{\mu}_j}.
  \end{equation}
  with $\alpha$, $\beta$, $\gamma$ as defined in \eqref{eq:param4Dj} and \eqref{eq:alphaBeta}, respectively, and 
  \begin{equation*}
   \lbar{\mu}\assign \min_j\lbar{\mu}_j.
  \end{equation*}
   Let $\eta \in [0, \lbar{\mu}]$ be maximal such that
  \begin{equation}\label{eq:assumptLambda}
   \lambda^e(\eta)>\lambda^i(\eta),
  \end{equation}
where
\begin{align}\label{eq:thm4lexlin}
\begin{split}
  \lambda^e_{\eta}  = \min_j \lambda_j^e(\eta), \quad \text{ with }\quad   &\lambda_j^e(\eta) =\frac{\mathfrak{w}_j(\eta)+\sqrt{\lbar{D}_j(\eta)}}{2\ubar{\mathfrak{a}}_j},\\ 
   \lambda^i_{\eta}  = \max_j \lambda_j^i(\eta), \quad  \text{ with }\quad  &\lambda_j^i(\eta) =\frac{\ubar{\mathfrak{b}}_j(\eta)}{\ubar{\mathfrak{a}}_j\; \lambda_j^e(\eta)}
\end{split}
\end{align}
with $\mathfrak{w}_j(\eta)$, $¸\ubar{\mathfrak{b}}_j(\eta)$, $\ubar{\mathfrak{a}}$ and $D_j(\eta)$ as defined in \eqref{eq:Djmu} and \eqref{eq:param4Dj}. Further, let $\sigma\in \left[0,\lbar{\mu}\right]$ be the largest value such that for all $j=1,\dots,n$
\begin{equation}\label{eq:assumptSigma}
 \that_j(\intval{s}_{\sigma}) + \left[-1, \, 1\right]\,\lambda^i_{\sigma}\, v_j \subseteq \intval{x}_{j}
\end{equation}
for $\intval{s}_{\sigma} \assign \left[p-\sigma y, \, p+\sigma y\right]$.
If\begin{equation}\label{eq:muFinal}
      \mu \assign \min(\eta, \sigma) > 0,
    \end{equation} 
then for all $s\in \interior(\widetilde{\intval{s}})\cap \intval{s}$ with 
\begin{equation*}
    \widetilde{\intval{s}} \assign \left\lbrace s\mid \absv{s-p}\leq \mu\: y\right\rbrace
    \end{equation*}
    there exists at least one solution $x$ of \eqref{eq:H} which lies inside the inclusion box 
    $\mathbf{R}^i_s$ (as defined in \eqref{eq:inclusionRegS}) and there are no solutions in $\cup_{s\in \intval{s}}\interior(\mathbf{R}^e_s)\setminus \cup_{s\in \intval{s}}(\mathbf{R}^i_s)$.
  \end{theorem}
  \begin{proof}
  Wlog., let $s=p+ \nu\, y\in\interior(\widetilde{\intval{s}})\cap \intval{s}$, i.e., $0<\nu<\mu$. In order to meet all requirements from Thm.~\ref{thm:pinPex}, which provides the result about the inclusion/exclusion regions at $s$, we have to check that the following three conditions hold:
  \begin{enumerate}[ref=(\roman*)]
   \item $D_j(s)>0$ for all $j=1,\dots,n$, with $D_j$ as in \eqref{eq:discT} (\emph{positivity requirement}). \label{req1}
   \item  $\lambda_s^i< \lambda_s^e$ (\emph{monotonicity of the inclusion/exclusion parameters}).\label{req2}
   \item $\left(\that_j(s) +\left[-1,\, 1\right]\, \lambda^i_s\,\:v_j\right) \subseteq \intval{x}_j$  $\forall \, j=1, \dots,n$ (\emph{feasibility of the inclusion region} $\mathbf{R}_s^i$). \label{req3}
  \end{enumerate}
  Condition \ref{req1} is satisfied by construction, since we get for $D_j(s)$ as in \eqref{eq:discT} by the calculations preceeding the statement of the theorem
  \begin{equation*}
  0 \leq \lbar{D}_j\left(\lbar{\mu}\right) \leq \lbar{D}_j\left(\mu\right)<\lbar{D}_j\left(\nu\right)\leq D_j\left(s\right)
  \end{equation*}
for all $s\in \interior(\widetilde{\intval{s}})\cap\intval{s}$.

For \ref{req2} we consider $\lambda^e_j$, $\lambda^i_j$ componentwise as functions in $\nu$, i.e.,
  \begin{equation}\label{eq:lambdaMu}
\lambda_j^e(\nu) \assign \frac{\lbar{\mathfrak{w}}_j(\nu)+\sqrt{\lbar{\mathfrak{w}}_j(\nu)^2-4\,\ubar{\mathfrak{a}}_j\ubar{\mathfrak{b}}_j(\nu)}}{2\,\ubar{\mathfrak{a}}_j}, \quad
\lambda_j^i(\nu) \assign \frac{\ubar{\mathfrak{b}}_j(\nu)}{\ubar{\mathfrak{a}}\; \lambda_j^e(\nu)}.
  \end{equation}  
Since by construction $\lbar{\mathfrak{w}}\leq \mathfrak{w}(s)$, $\ubar{\mathfrak{b}}\geq \ubar{\mathfrak{b}}(s)$, $\ubar{\mathfrak{a}}\geq \mathfrak{a}(s)$, we have 
\begin{equation}\label{eq:prf1}
\lambda_j^e(\nu) \leq \lambda_j^e(s), \quad \lambda_j^i(\nu) \geq \lambda_j^i(s)
\end{equation}
with $\lambda_j^e(s)$, $\lambda_j^i(s)$ as in \eqref{eq:lambdajS}. Both $\lambda^e_j\left(\nu\right)$ and $\lambda^i_j\left(\nu\right)$ are depending continuously on $\nu$. In particular, for increasing $\nu$,  $\lambda^e_j(\nu)$ is monotonically decreasing and $\lambda^i_j(\nu)$ monotonically increasing. Hence, $\lambda^e_{\nu} = \min_j \lambda_j^e(\nu)$ and $\lambda^i_{\nu} = \max_j \lambda_j^i(\nu)$ have the same monotonicity behaviour, since we take a minimum (resp.~maximum) of a monotonically decreasing (resp.~increasing) function.
By computing $\lbar{\mu}$ from \eqref{eq:lbarMuj}, we get a lower and an upper bound on the exclusion and inclusion parameter $\lambda^e$ and $\lambda^i$, respectively, since  
$\lbar{D}_k\left(\lbar{\mu}\right) = 0$ implies $\lambda_k^e\left(\lbar{\mu}\right)=\lambda_k^i\left(\lbar{\mu}\right)$ for some $k\in \lbrace{1,\dots,n\rbrace}$. Since we choose $\eta\in \left[0,\, \lbar{\mu}\right]$ in such a way that $\lambda^e_{\eta}>\lambda^i_{\eta}$, we have in particular
\begin{equation*}
 \lambda_j^e(\eta)\geq \lambda_k^e(\lbar{\mu}), \quad \lambda_j^i(\eta)\leq \lambda_k^i(\lbar{\mu})
 \end{equation*}
for all $j = 1,\dots,n$. By monotonicity, we have 
\begin{equation}\label{eq:prf2}
\lambda_j^e(\nu)\geq \lambda_j^e(\eta), \quad \lambda_j^i(\nu)\leq \lambda_j^i(\eta).
\end{equation}
Taking the minimum (resp.~maximum) over all $j$, we get by \eqref{eq:prf1}, \eqref{eq:prf2} and assumption \eqref{eq:assumptLambda}
\begin{equation*}
 \lambda^i_s \leq \lambda^i_{\nu}\leq \lambda^i_{\eta}<\lambda^e_{\eta}\leq\lambda^e_{\nu}\leq \lambda^e_s,
\end{equation*}
hence, condition \ref{req2} is satisfied.

Finally, we have for $\nu\leq \sigma$ and by assumption \eqref{eq:assumptSigma}
\begin{equation*}
 \that_j(s)+\lambda_s^i\:v_j  \leq \ubar{\that_j(\intval{s}_{\nu})}+\lambda^i_{\nu}\:v_j\leq \ubar{x}_j, \quad \text{ and }\quad 
  \that_j(s)-\lambda_s^i\:v_j  \geq \lbar{\that_j(\intval{s}_{\nu})}-\lambda^i_{\nu}\:v_j\geq \lbar{x}_j,
\end{equation*}
since $\that(s)\in \that(\intval{s}_{\nu})$, and $\lambda_s^i\leq \lambda_{\nu}^i$ by \ref{req2}. Hence, condition \ref{req3} is satisfied as well, which concludes the proof. 
 \end{proof}
 As for the non-parametric case, the above considerations simplify if \eqref{eq:H} is quadratic in both $x$ and $s$. Since the first order slopes then are linear in $x$ and $s$, all second order slopes are constant. Hence the estimates $\mathfrak{B}(x,s)$ become valid everywhere in the domain of definition of $H$, i.e., $\ubar{\mathfrak{B}}=\mathfrak{B}(x,s)=\ubar{B}$. If the approximation function \eqref{eq:that} is linear, its first order slope \eqref{eq:slope} is constant, which simplifies the bounds in \eqref{eq:G0thm} and \eqref{eq:Athm}. In particular, if \eqref{eq:H} is quadratic and the approximation function $\that$ is linear both in $x$ and $s$, $\ubar{A}$ from \eqref{eq:Athm} is constant.

\section{Example}\label{sec:examples}

\paragraph{Inclusion/exclusion region for fixed parameter $p$.}
	We consider the system of equations
	\begin{equation}\label{eq:exH}
	H(x,s) = \begin{pmatrix}
	x_1^2 + x_2^2-26+s^2\\
	x_1\cdot x_2 -13+s
	\end{pmatrix} = 0
	\end{equation}
	for $s\in \intval{s} = [0,\;2]$, $x \in \intval{x}= [0,\;5]\times [0,\; 5]$.
	
	We set $p = 1$ and compute a corresponding solution $z = (3,4)^T$.
	A slope for $H$ with center $(z,p)$ can be computed as
	\begin{align}\label{eq:exHslope}
	H(x,s)-H(z,p) & = \underbrace{\begin{pmatrix}
		x_1+z_1 & x_2 + z_2 & s+p\\
		x_2 & z_1 & 1
		\end{pmatrix}}_{=\slope[(z,p),(x,s)]{H}}
	\begin{pmatrix}
	x_1-z_1\\
	x_2-z_2\\
	s-p
	\end{pmatrix}.
	\intertext{We get for the solution from above}
	\begin{split}
	\slope[(z,p), (x,s)]{H} &= 
	\begin{pmatrix}
	\slope[(z,p), (x,s)]{H_x}, &  \slope[(z,p), (x,s)]{H_s}
	\end{pmatrix}\\\nonumber
	& = \begin{pmatrix}	
	\begin{pmatrix}
	x_1+3 & x_2 +4 \\
	x_2 & 3
	\end{pmatrix}, 
	\begin{pmatrix}
	s+1\\
	1
	\end{pmatrix}
	\end{pmatrix},     
	\end{split}
	\end{align}
	and for the Jacobian of $H$ wrt.~$x$ at $(z,p)$
	\begin{equation*}
	H'_{x} = \begin{pmatrix}
	2x_1 & 2x_2\\x_2 & x_1
	\end{pmatrix}, \quad
	H'_{x}(z, p) = \begin{pmatrix} 6 & 8 \\ 4& 3\end{pmatrix}.
	\end{equation*}
	For the preconditioning matrix $C$ we take
	\begin{equation}\label{eq:exC}
	C \assign H'_{x}(z,p)^{-1} = \frac{1}{14}\begin{pmatrix*}[r]
	-3 & 8\\ 4 & -6
	\end{pmatrix*}.
	\end{equation}
	The slope from \eqref{eq:exHslope} can be put in form \eqref{eq:FslopeEst2} with
	\begin{equation*}
	H_1 = \begin{pmatrix} 1 & 0 & 0  \\ 0 & 0 & 0\end{pmatrix}, \quad H_2 = \begin{pmatrix}
	0 & 1& 0 \\ 1 & 0 & 0
	\end{pmatrix}, \quad H_3 = \begin{pmatrix}
	0 & 0 & 1\\ 0 & 0 & 0
	\end{pmatrix}.
	\end{equation*}
	Thus, we get for the bounds from \eqref{eq:bounds0}
	\begin{equation}\label{eq:exBounds0}
	\ubar{b} = (0,0)^T, \quad B_0 = \begin{pmatrix}
	0 & 0 \\ 0 & 0
	\end{pmatrix}, \quad \ubar{B}= \begin{pmatrix}
	\frac{1}{14}
	\begin{pmatrix}
	3 & 0 \\ 4 & 0
	\end{pmatrix}, \;
	\frac{1}{14}
	\begin{pmatrix}
	8 & 3 \\ 6 & 4
	\end{pmatrix}
	\end{pmatrix},
	\end{equation}
	where $\ubar{b}$ and $B_0$ both vanish, since $z$ happens to be an exact zero of \eqref{eq:exH}, and we computed without roundoff-errors.
	With $v=(1,1)^T$ we get from \eqref{eq:aw}
	\begin{equation}\label{eq:exaw}
	w = (1,1)^T, \quad a = \left(\ubar{B}_1+\ubar{B}_2\right)v = (1,1)^T,
	\end{equation}
	and with $D_1 = D_2 = 1$ we get from \eqref{eq:lambdaExIn}
	\begin{equation*}
	\lambda^e = 1, \quad \lambda^i = 0.
	\end{equation*}
	Hence, we get 
	\begin{equation*}
	\intval{R}^e_0 = \begin{pmatrix}
	[2,4]\\
	[3,5]
	\end{pmatrix}, \quad \intval{R}^i_0 = \begin{pmatrix} 3\\4\end{pmatrix}.
	\end{equation*}

\paragraph{Construction of a feasible parameter interval $\widetilde{\intval{s}}$.} We consider a linear approximation function
  \begin{equation*}
    \that \colon \zz{R}\to \zz{R}^2, \quad \that(s) = z+\Theta\, (s-p)
  \end{equation*}
with $\Theta \in \zz{R}^{2\times 1}$. We compute the parameter interval $\widetilde{\intval{s}}$ from Thm.~\ref{thm:compMu} for two different approximation functions, 
\begin{enumerate}
 \item a \emph{tangent} $\that^{tan}$ in $(z,p)$ with 
 \begin{equation}\label{eq:ThetaTang}
 \Theta^{tan} = -(H'_x(z,\, p))^{-1}\,(H'_s(z, p)) = -\frac{1}{7}\begin{pmatrix}1\\1 \end{pmatrix},
 \end{equation}
 \item a \emph{secant} $\that^{sec}$ through the center $(z,p)$ and a second point $x^1=(\sqrt{13},\sqrt{13})^T$ at $s^1=0$ with
 \begin{equation}\label{eq:ThetaSec}
  \Theta^{sec} = \left(x^1-z\right)\, \left(s^1-p\right)^{-1} = \begin{pmatrix}
                                                                    3-\sqrt{13}\\4-\sqrt{13}
                                                                   \end{pmatrix}.
\end{equation}
\end{enumerate}
Thus, we have
\begin{equation*}
g(s) = \begin{pmatrix}
           z+\Theta \, (s-p)\\
           s
         \end{pmatrix}\quad \text{ with constant slope matrix } \slope[s,p]{g}=\begin{pmatrix}
									 \Theta\\1\end{pmatrix}.
\end{equation*}
In order to apply Thm.~\ref{thm:compMu} we compute the upper bounds $\ubar{G}_0$ and $\ubar{A}$ from \eqref{eq:G0thm} and \eqref{eq:Athm}, respectively. A slope for $H'_x$ is 
  \begin{equation*}
    \slope[g(p), g(s)]{(H'_x)} = \begin{pmatrix}
      \begin{pmatrix}
	2 & 0 \\ 0 & 1
      \end{pmatrix}, \;
      \begin{pmatrix}
	0&2\\1&0
      \end{pmatrix}, \;
      \begin{pmatrix}
	0 & 0\\
	0 & 0
      \end{pmatrix}
    \end{pmatrix},                                                  
  \end{equation*}
since
  \begin{align*}
   H'_x(x,s) - H'_x(z, p) & = \begin{pmatrix}
      2 & 0 \\ 0&  1
    \end{pmatrix}(x_1-z_1) + \begin{pmatrix}
      0 & 2 \\ 1 & 0\end{pmatrix}(x_2-z_2).
  \end{align*}
With preconditiong matrix $C$ from \eqref{eq:exC} we get
\begin{align*}
 \ubar{G}_0^{tan} & = \frac{1}{98}\begin{pmatrix}51\\58\end{pmatrix}, &\text{ and }\quad  \ubar{A}^{tan} &= \frac{1}{7}\begin{pmatrix}1&1\\1&1\end{pmatrix},\\
  \ubar{G}_0^{sec} & = \frac{1}{7}\begin{pmatrix}14(\sqrt{13}-3)\\58-14\sqrt{13}\end{pmatrix},  &\text{ and }\quad \ubar{A}^{sec} &= \frac{1}{7}\begin{pmatrix}7-\sqrt{13}&\sqrt{13}\\\sqrt{13}&7-\sqrt{13}\end{pmatrix}
\end{align*}
for tangent $\that^{tan}$ and secant $\that^{sec}$, respectively. Thereby we compute the tensor-vector-product in the formula for $\ubar{A}$ using the tensor rules \eqref{eq:tensorRules}. Since we have a quadratic problem, the second order slopes are constant for all $x \in \intval{x}$, $s\in \intval{s}$, and thus, $\ubar{\mathfrak{B}}= \ubar{B}$ from \eqref{eq:bounds0}. With these preparations, we are able to compute $\mu$. We take $y = 1$, and get from \eqref{eq:alphaBeta}
\begin{enumerate}
 \item for tangent $\that^{tan}$: 
 \begin{equation*}
  \alpha^{tan} = \frac{2}{7}\begin{pmatrix}1\\1\end{pmatrix},\quad \beta^{tan}=\frac{1}{49}\begin{pmatrix}65\\72\end{pmatrix},\quad \gamma^{tan}=\begin{pmatrix}1\\1\end{pmatrix},
 \end{equation*}
\item and for secant $\that^{sec}$
  \begin{equation*}
  \alpha^{sec} = \begin{pmatrix}1\\1\end{pmatrix},\quad \beta^{sec}=\frac{1}{7}\begin{pmatrix}28\sqrt{13}-77\\123-28\sqrt{13}\end{pmatrix},\quad \gamma^{sec}=\begin{pmatrix}1\\1\end{pmatrix},
 \end{equation*}
\end{enumerate}
which results in
\begin{equation*}
 \mu^{tan} \approx 0.343, \quad\quad \mu^{sec}\approx 0.149.
\end{equation*}

  \begin{figure}[htb]
      \captionsetup{width=0.7\textwidth}
  	\centering
    {\includegraphics[width=0.85\textwidth ]{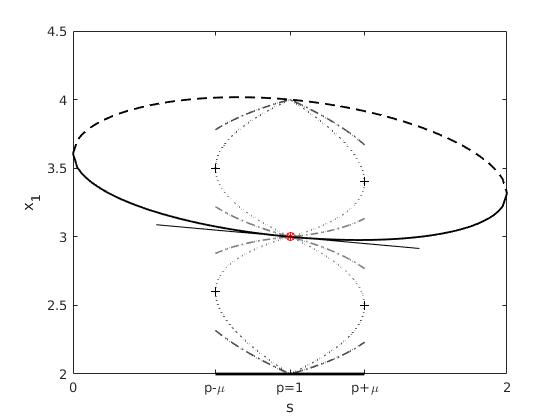}}
    \caption[]{Two solution curves (solid and dashed lines) for $x_1$ over $\intval{s}$, approximation function $\that^{tan}$ and inclusion regions (expanding from center $*$) and exclusion regions (contracting towards the solution curve (solid line)).
    }
    \label{fig:exInExBox}
   
  \end{figure}

The respective inclusion and exclusion parameters from \eqref{eq:thm4lexlin} are
\begin{align*}
 \lambda^i_{\mu^{tan}} \approx 0.45092049, \quad & \lambda^e_{\mu^{tan}} \approx 0.45092053\\
 \lambda^i_{\mu^{sec}} \approx 0.42531789, \quad & \lambda^e_{\mu^{sec}} \approx 0.42531794,
\end{align*}
so, for both approximation functions the monotonicity requirement \eqref{eq:assumptLambda} holds for $\mu$. We still have to check the feasibility condition \eqref{eq:assumptSigma} of the inclusion region for the parameter intervals
 \begin{equation}\label{eq:exIntvalS}
  \intval{s}_{\mu}^{tan}=\left[0.657,\, 1.343\right], \quad \intval{s}_{\mu}^{sec} = \left[0.851,\, 1.149\right].
 \end{equation}
For both intervals the feasibility requirement is met, since
 \begin{align}\label{eq:exXboxes}
  \begin{split}
   \hat{\intval{x}}(\intval{s}_{\mu}^{tan}) = \that(\intval{s}_{\mu}^{tan})+[-1,\, 1]\,\lambda^i_{\mu^{tan}}\, v&= \begin{pmatrix}[2.406,\, 3.594]\\ [3.406,\, 4.594]\end{pmatrix}\subseteq \begin{pmatrix}[0,\, 5]\\ [0,\, 5]\end{pmatrix}
   \\
    \hat{\intval{x}}(\intval{s}_{\mu}^{sec}) = \that(\intval{s}_{\mu}^{sec})+[-1,\, 1]\,\lambda^i_{\mu^{sec}}\, v &= \begin{pmatrix}[1.969,\, 4.031]\\ [3.180,\, 4.820]\end{pmatrix}\subseteq \begin{pmatrix}[0,\, 5]\\ [0,\, 5]\end{pmatrix}.
 \end{split}
 \end{align}
Hence, for both parameter intervals from \eqref{eq:exIntvalS} the existence of at least one solution of \eqref{eq:exH} can be guaranteed. The boxes from \eqref{eq:exXboxes} provide a first outer approximation of the solution set. As we could already see in this low dimensional example, the choice of the approximation function greatly influences the size of the computed parameter box as well as the quality of the enclosure of the solution set. A good choice of $\that(s)$ is thus important, and  may require a closer analysis of the problem at hand.

In Fig.~\ref{fig:exInExBox} two solution curves for $x_1$ over $\intval{s}$ are shown together with the approximation tangent $\that^{tan}(s)$ in $(z,p)$ and corresponding inclusion and exclusion regions. The dash-dotted lines represent inclusion (exclusion) regions, which are computed using the true bounds at each parameter value $s$, i.e., with $G_0(s)$ and $A(s)$ from \eqref{eq:obarbnew} and \eqref{eq:BzeroNew}. The dotted lines represent the developement of the inclusion (exclusion) exclusion regions which are calculated using upper bounds over the initial interval $\intval{s}$, i.e., with $\ubar{G_0}$ and $\ubar{A}$ from \eqref{eq:upobarb} and \eqref{eq:upBzero}. The latter ones increase (decrease) much faster and intersect at the boundary of $\widetilde{\intval{s}}$. In Fig.~\ref{fig:exInExBox_comp} these inclusion and exclusion regions as well as a comparison between the respective regions for the tangent and the secant are depicted.
  \begin{figure}[htb]
    \centering
    \subfloat[]{\includegraphics[width=0.5\textwidth ]{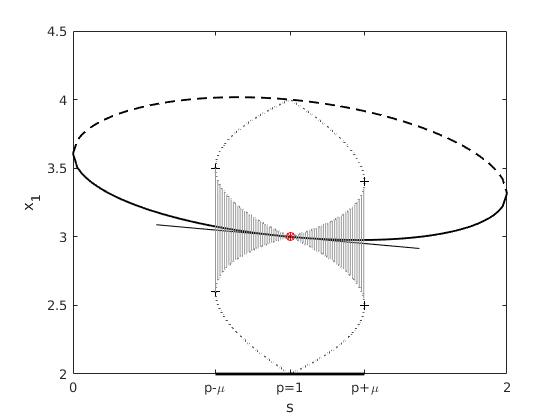}}
    \subfloat[]{\includegraphics[width=0.5\textwidth ]{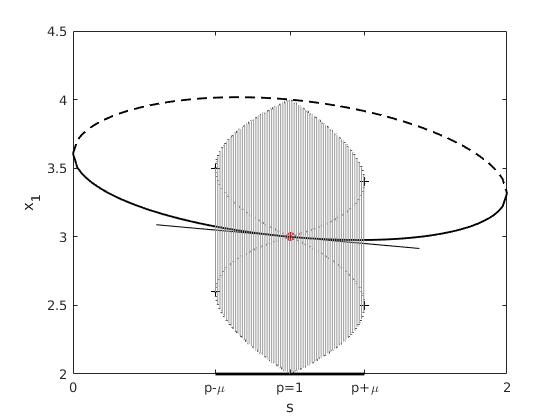}} \\
    \subfloat[]{\includegraphics[width=0.5\textwidth ]{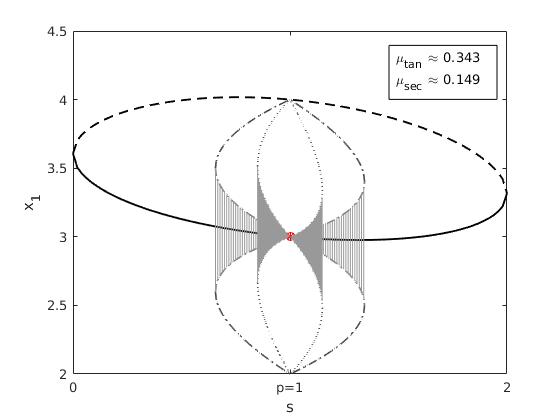}}
    \caption[]{Inclusion (a) and exclusion (b) regions over $\widetilde{\intval{s}}^{tan}\approx[0.657,\,1.343]$ for $\that^{tan}$ (computed with upper bounds from \eqref{eq:upobarb} and \eqref{eq:upBzero}), and (c) in comparison with the respective regions over $\widetilde{\intval{s}}^{sec}\approx[0.851,\,1.149]$ for $\that^{sec}$}
    \label{fig:exInExBox_comp}
  \end{figure}

\section{Future Work}\label{sec:discussion}
The above described method allows to explictely construct feasible areas in the parameter space and to rigorously enclose the solution set of \eqref{eq:H} for all parameters in the computed parameter boxes. The method shows promising applicability for problems in low parameter dimensions. However, the method suffers from some sort of cluster effect when approaching the boundaries of the feasible parameter regions, i.e., the step size $\mu$ and thus the radii of the parameter boxes become smaller and smaller. This problem may be tackled using an extension of Miranda's theorem and is addressed in \cite{ponleitnerSchichl2018mir}.
An application for the new method is for example the workspace-computation of parallel manipulators (see \citet{merlet06}). In particular, the computation of the total orientation workspace requires the solution of a parameter-dependent system of nonlinear equations. Up to the author's knowledge, there are only a few results adressing this problem using rigorous methods (\citet{merlet1998ws, merlet1998}). Therefore, the proposed method will be applied to the workspace problem \cite{ponleitnerSchichl2018ws}.
 
  \clearpage
  \small
  \renewcommand\bibname{References}
  \bibliographystyle{plainnatnew1}
  \bibliography{paperExclusionRegions}
\end{document}